\documentclass{birkjour}

\usepackage{amssymb}
\usepackage{amsmath}
\usepackage{amsthm}
\usepackage{mathtools}

 \newtheorem{thm}{Theorem}[section]
 \newtheorem{cor}[thm]{Corollary}
 \newtheorem{lem}[thm]{Lemma}
 \newtheorem{prop}[thm]{Proposition}
 \theoremstyle{definition}
 
 \theoremstyle{remark}

 \numberwithin{equation}{section}

\newcommand{\Lip}{\operatorname{Lip}}

\newcommand{\Id}{\operatorname{Id}}

\begin{document}

\title[Lipschitz vs Linear Numerical Index]{Lipschitz vs Linear Numerical Index\\ in certain Banach spaces}

\author[Antonio P\'{e}rez-Hern\'{a}ndez]{Antonio P\'{e}rez-Hern\'{a}ndez}

\address{%
Departamento de Matem\'{a}tica Aplicada I\\
Escuela T\'{e}cnica Superior de Ingenieros Industriales\\
Universidad Nacional de Educación a Distancia\\
28040 Madrid (Spain)\\[1mm]
Orcid: \texttt{0000-0001-8600-7083}}

\email{antperez@ind.uned.es}



\date{January 1, 2004}

\thanks{ This work was partially supported by a PhD fellowship of Fundaci\'{o}n ``la Caixa'' (ID 100010434), ref. LCF/BQ/DE13/10300005. }

\begin{abstract}
We show that for real Banach spaces that are either separable or dual spaces, the Lipschitz numerical index coincides with the classical (linear) numerical index. This result provides partial evidence toward the question posed by Wang, Huang, and Tan (2014) of whether these two quantities coincide for every real Banach space. Our approach relies on two standard linearization techniques for Lipschitz maps: differentiation via convolution with Gaussian probability measures, and  invariant means.
\end{abstract}

\maketitle

\section{Introduction}

The theory of numerical index of Banach spaces has its roots in the theory of numerical range of matrices and operators between Hilbert spaces. In the 1950s, it gained further significance through applications in Banach algebra theory, which led to its extension to general Banach spaces in independent works by Lumer \cite{Lumer61} and Bauer \cite{Bauer62}. Since then, the theory has evolved and consolidated into a distinct area of functional analysis, with connections to operator theory and geometry of Banach spaces. We refer to \cite{BoDu71, BoDu73, GustafsonRao97,Martin00, Kadets06, CaPa14, Kadets20} for further details and references.

Let $X$ be a real or complex Banach space, and denote by $X^{\ast}$ its dual space. We will use the notation $x^{\ast}(x)  = \langle x^{\ast},x \rangle = \langle x,x^{\ast} \rangle $ for every $x \in X$ and $x^{\ast} \in X^{\ast}$. We write $B_{X}$ and $S_{X}$ for its closed unit ball and its unit sphere, respectively. Let $\mathcal{L}(X)$ denote the space of all bounded linear operators on $X$. For each $T \in \mathcal{L}(X)$, the \emph{numerical range} of $T$ is defined by
\[
W(T) := \left\{ \langle x^{\ast}, Tx \rangle \colon x \in S_{X} \,, \, x^{\ast} \in S_{X^{\ast}} \,, \, x^{\ast}(x) = 1 \right\},
\]
and its \emph{numerical radius} by
\[
w(T) := \sup{\{ |\lambda| \colon \lambda \in W(T) \}}.
\]
so that $w(T) \leq \| T\|$. The \emph{numerical index} of the Banach space $X$ is then defined by
\[
\begin{split} 
n(X) & = \inf{\{ w(T) \colon T \in \mathcal{L}(X)\,,\, \| T\| = 1 \}}\\[1mm]
&  = \max{\{ \lambda \geq 0 \colon \lambda \| T\| \leq w(T) \mbox{ for every } T \in \mathcal{L}(X) \}}.  
\end{split}
\]
We can also understand this constant as a measure of rotundity around the identity operator $\Id_{X}$ in the unit sphere of $\mathcal{L}(X)$, since $n(X)$ is the largest constant $\alpha \geq 0$ such that
\[ \max_{ |\theta|=1} \| \operatorname{Id}_{X} + \theta T\| \geq 1 + \alpha \| T\| \]
for every operator $T$, see e.g. \cite[Proposition 3.3]{Kadets20}. In particular, if $n(X)>0$ then $\operatorname{Id}_{X}$ is a strongly extreme point of the unit ball of $\mathcal{L}(X)$. 

These notions have been extended in several directions. In this work, we focus on a non-linear variant introduced by X. Huang, D. Tan, and R. Wang in \cite{Wang14}, known as the \emph{Lipschitz numerical index}. Let $\Lip_{0}(X)$ denote the Banach space of all Lipschitz maps $F:X \longrightarrow X$ such that $F(0) = 0$, equipped with the norm
\begin{equation*} 
\| F\|_{L}:=\sup{\left\{ \frac{\| F(x_{1}) - F(x_{2})\|}{\| x_{1} - x_{2}\|} \colon x_{1}, x_{2} \in X, x_{1} \neq x_{2} \right\}}< \infty.
\end{equation*}
 For every $F \in \Lip_{0}(X)$, the \emph{numerical range} of $F$ is defined by 
\begin{multline*}
W_{L}(F) = \left\{ \frac{\langle x^{\ast}, F(x_{1}) - F (x_{2}) \rangle}{\| x_{1} - x_{2}\|} \colon  \, x^{\ast} \in S_{X^{\ast}}\,, \, x_{1}  \neq  x_{2} \in X\,, \,\right.\\
\left.\frac{\langle x^{\ast}, x_{1} - x_{2} \rangle}{\| x_{1} - x_{2}\| } = 1 \right\},
\end{multline*}
and its \emph{numerical radius} by
\[
w_{L}(F) = \sup{\{ |\lambda| \colon \lambda \in W_{L}(F) \}}.
\]
The \emph{Lipschitz numerical index} of the Banach space $X$ is then defined by
\[
\begin{split} 
n_{L}(X) & = \inf{\{ w_{L}(F) \colon F \in \Lip_{0}(X), \| F\|_{L} = 1 \}}\\[1mm]
&  = \max{\{ \lambda \geq 0 \colon \lambda \| T\|_{L}   \leq w_{L}(T) \mbox{ for every } T \in \Lip_{0}(X) \}}.  
\end{split}
\]
Notice that the canonical inclusion $\mathcal{L}(X) \hookrightarrow \Lip_{0}(X)$ is a linear isometry, and that if $F$ is linear $W_{L}(F) = W(F)$ and  $w_{L}(F) = w_{L}(F)$. Therefore, 
\[ n_{L}(X) \leq n(X)\,. \]
A central question posed in \cite{Wang14} is whether equality always holds. The answer to this question depends on whether the underlying field is $\mathbb{R}$ or $\mathbb{C}$. 

\subsubsection*{Complex case}

In the complex case, the two indices may differ. A simple example is given by the complex Hilbert space $X=\mathbb{C}^{2}$ with the usual norm
\[ \| (z_{1}, z_{2}) \| = \sqrt{|z_{1}|^2 + |z_{2}|^2}\,. \]
It is well-known that its (linear) numerical index is  $n(\mathbb{C}^{2}) = 1/2$. In fact, this holds for every Hilbert space of dimension larger than two, see \cite{Martin00}. However, its Lipschitz numerical index is $n_{L}(\mathbb{C}^2)=0$. To see this, consider the map $F: \mathbb{C}^{2} \rightarrow \mathbb{C}^{2}$ given by
\[ F(z_{1}, z_{2}) = (\overline{z_{2}}, - \overline{z_{1}}), \]
which is $\mathbb{R}$-linear (not $\mathbb{C}$-linear) and Lipschitz with $\| F\|_{L} = 1$.  The fact that $F$ is $\mathbb{R}$-linear yields that its Lipschitz numerical range can be actually rewritten as
\[
W_{L}(F) = \left\{ \langle x^{\ast}, F(y) \rangle \colon  \, x^{\ast} \in S_{X^{\ast}}\,, \, y \in S_X\,, \, \langle x^{\ast}, y \rangle = 1 \right\}\,.
\]
Given any $y=(w_{1},w_{2}) \in S_{X}$ with $\| y\|=1$, there is a unique functional $x^{\ast} \in S_{X^\ast}$ with $x^{\ast}(y)=1$, namely
\[ x^{\ast}(z_{1}, z_{2}) = \overline{w}_{1}z_{1} + \overline{w}_{2} z_{2}\,. \]
We then compute
\[ \langle x^{\ast} , F(y) \rangle = x^{\ast}(\overline{w}_{2}, - \overline{w}_{1}) = \overline{w}_{2} \overline{w}_{1} - \overline{w}_{1} \overline{w}_{2} = 0\,. \]
Thus, $W_{L}(F) = \{ 0\}$, so its numerical radius is $w_{L}(F) =0$, which implies that $n_{L}(\mathbb{C}^2) = 0$.

Kadets et al. \cite{Kadets15} have demonstrated that for (complex) \emph{lush spaces} both indices coincide and are equal to one. In particular, if the Banach space $X$ is \emph{strongly regular} (e.g., RNP spaces) or does not contain copies of $\ell_{1}$ (e.g., Asplund spaces), the condition $n(X)=1$ implies that $n_{L}(X) = 1$, as these spaces are lush by the results of \cite{Aviles10}. Whether this implication holds in full generality remains an open question.

Furthermore, it is known that for complex Banach spaces, the set of possible values for $n(X)$ is the interval $[e^{-1}, 1]$, see \cite{DMPW70}. On the other hand, the set of possible values for the Lipschitz numerical index appears to be $[0, 1]$. This follows from \cite[Proposition 2.2]{Choi25}, noting that the Lipschitz numerical index of $(\mathbb{C}^2, \| \cdot \|_{\infty})$ is one (by the aforementioned result of Kadets et al.), while the Lipschitz numerical index of $(\mathbb{C}^2, \| \cdot \|_{2})$ is zero, as shown above. It is therefore natural to ask which values are realized by the pair $(n(X), n_{L}(X))$ as $X$ ranges over all complex Banach spaces. However, we do not pursue this question in the present work.

\subsubsection*{Real case}

In contrasts to the complex case, no known example shows that $n_{L}(X)$ and $n(X)$ may be different when $X$ is a real space. This motivates the conjecture $n_{L}(X)=n(X)$ for every real Banach space $X$. In fact, partial results support this claim:  Wang et al.  proved that if $X$ has the Radon-Nikod\'ym property (RNP) then the equality holds; moreover, they showed that if $X$ is a real lush space, then both indices coincide and are equal to one. In particular, if  $X$ does not contain a copy of $\ell_{1}$ or $X$ is \emph{strongly regular}, then the condition $n(X)=1$ yields that $X$ is a (real) \emph{lush space}, as proved in \cite{Aviles10}, and thus $n_{L}(X)=1$ by \cite[Theorem 2.6]{Wang14}.

In this paper, we contribute further evidence to this conjecture. We prove that the equality $n(X) = n_{L}(X)$ holds if $X$ is a (real) separable Banach space (Section \ref{sec:realSeparable}) or if $X$ is a (real) dual Banach space (Section \ref{sec:realDual}). The arguments are based on two distinct techniques for constructing linear operators from Lipschitz maps that preserve certain properties: differentiation via convolution with a Gaussian probability measure, and the use of invariant means.


\section{Real separable Banach spaces}\label{sec:realSeparable}

We first state the main result of the section.

\begin{thm}\label{Theo:EqualityNumIndexesSeparable}
Let $X$ be a real separable Banach space. Then, $n(X) = n_{L}(X)$.
\end{thm}

Prior to the proof, we need to recall some basics on weak Gateaux differentiability and convolution with Gaussian measures of Lipschitz maps. 

Let $F:X \to X$ be a Lipshitz map. We say that $F$ is \emph{weakly Gateaux differentiable} at $x_{0} \in X$ if there is a linear map $T: X \to X$ such that
\[ \langle x^{\ast},  T x \rangle = \lim_{t \rightarrow 0}{\frac{\langle x^{\ast}, F(x_{0} + t x) -  F(x_{0}) \rangle}{t}} \]
for every $x \in X$ and $x^{\ast} \in X^{\ast}$. In this case, the operator $T$ is unique and will be denoted by $dF[x_{0}]$. The numerical index of $F$ and that of its weak derivatives can be related as follows.

\begin{lem}\label{Lemm:GateauxNumericalRadius} 
Let $X$ be a real Banach space and let $F: X \to X$ be weakly Gateaux differentiable at $x_{0} \in X$. Then
\begin{equation}\label{equa:GateauxNumericalRadius} 
\| dF[x_{0}] \| \leq \| F\|_{L} \quad \quad \text{and} \quad \quad w\left(dF[x_{0}]\right) \leq w_{L}(F) \,.
\end{equation}
\end{lem}
\begin{proof}
Recall that for every $x^{\ast} \in S_{X^{\ast}}$ and $x \in S_{X}$
\begin{equation}\label{equa:GateauxNumericalRadius2} 
 \langle x^{\ast}, dF[x_{0}] (x) \rangle \, = \, \lim_{t \rightarrow 0} \,   , \frac{\langle x^{\ast}, F(x_{0} + t x) - F(x_{0}) \rangle }{\| (x_{0} + tx) - x_{0} \|}  \,.
\end{equation}
This yields that $\langle x^{\ast}, dF[x_{0}] (x) \rangle \leq \| F\|_{L}$ for every $x \in S_{X}$ and $x^\ast \in S_{X^\ast}$, from where the first inequality follows. Moreover, if $x^{\ast}(x)=1$, then for every $t \neq 0$
\[  1 = \langle x^{\ast}, x\rangle = \frac{\langle x^{\ast}, (x_{0}+tx) - x_{0} \rangle}{\| (x_{0}+tx) - x_{0} \|},\]
and thus
\[ \left| \frac{\langle x^{\ast}, F(x_{0} + t x) - F(x_{0}) \rangle }{\| (x_{0} + tx) - x_{0} \|} \right| \leq w_{L}(F)\,.\]
Taking the limit when $t \rightarrow 0^+$ in the previous expression, we deduce that
\[
 |\langle x^{\ast}, dF[x_{0}] (x) \rangle| \, \leq \, w_{L}(F)  \,.
 \]
Finally, taking then supremum over $x \in S_{X}$ and $x^\ast \in S_{X^\ast}$ with $x^{\ast}(x)=1$, the second inequality follows.
\end{proof}

Let us next recall some basic facts concerning Gaussian measures on Banach spaces. We refer to \cite[Chapter 6]{BenyaminiLindenstrauss00} for terminology and further details (see also \cite{Bogachev98}).

\begin{lem}\label{Lemm:measureDifferentiable}
Let $X$  be a real separable Banach space. Then, there exists a nondegenerate Gaussian probability measure $\gamma$ on $X$ with mean zero such that
\begin{equation}\label{equa:GaussianIntegratesCone} 
\int_{X}{\| u\| \: d \gamma(u)} < + \infty.
\end{equation}
\end{lem}


The next fundamental result shows that convoluting a Lipschitz map with a Gaussian measure gives a new Lipschitz map featuring better differentiability properties.

\begin{prop}\label{Prop:sufConditionWeakGateauxDiff}
Let $X$ be a real separable Banach space and let $\gamma$ be a Gaussian measure on $X$ as in Lemma \ref{Lemm:measureDifferentiable}.  If $F: X \rightarrow X$ is a Lipshitz operator, then the map $F_{\gamma}: X \rightarrow Y$ defined by
\[ F_{\gamma}(x) = (F \ast \gamma)(x) = \int_{X}{\left[F(x + u) - F(u)\right] \: d \gamma(u)} \,\,\, , \,\,\, x \in X \]
is a Lipschitz operator satisfying $F_{\gamma}(0)=0$,  $\| F_{\gamma}\|_{L} \leq \| F\|_{L}$ and it is weakly Gateaux differentiable at every $x \in X$. Moreover, $w_{L}(F_{\gamma}) \leq w_{L}(F)$.
\end{prop}

\begin{proof}
The first two properties, namely $F_{\gamma}(0)=0$ and  $\| F_{\gamma}\|_{L} \leq \| F\|_{L}$ are straightforward. The fact that $F_{\gamma}$ is weakly Gateaux differentiable at every $x \in X$ is claimed in \cite[p. 129]{GodefroyKalton03} and proved in  \cite[Theorem 7]{BogachevShkarin88}.  We then simply have to check the last inequality for the numerical radius. Let $x_{1} \neq x_{2}$ in $X$ and $x^{\ast} \in S_{X^{\ast}}$ with 
\[|\langle x^{\ast}, x_{1} - x_{2} \rangle | = \| x_{1} - x_{2}\|.\] 
Then, we can upper estimate
$$ \frac{|\langle x^{\ast}, F_{\gamma}(x_{1}) - F_{\gamma}(x_{2}) \rangle |}{\| x_{1} - x_{2}\|} \leq \int_{X}{\frac{|\langle x^{\ast}F(x_{1} + u) - F(x_{2} + u) \rangle|}{\| (x_{1}+u) - (x_{2}+u)\|} \: d \gamma(u)} \leq w_{L}(F) $$
since the integrand is pointwise bounded by $w_{L}(F)$ by definition and $\gamma$ is a probability measure. Finally, taking supremum in the precious inequality over $x^{\ast}, x_{1}, x_{2}$ under the above conditions we conclude that $w_{L}(F_{\gamma}) \leq w_{L}(F)$. finishing the proof.
\end{proof}

\noindent We are now prepared to prove the main result of the section.

\begin{proof}[Proof of Theorem \ref{Theo:EqualityNumIndexesSeparable}]
We only need to show that $n_{L}(X) \geq n(X)$. Let us fix an arbitray $\varepsilon \in (0,1)$. By the definition of $n_{L}(X)$, there is a Lipschitz operator $F: X \rightarrow X$ such that 
\[ F(0) = 0 \quad , \quad \| F\|_{L} = 1 \quad , \quad  w_{L}(F) < n_{L}(X) + \varepsilon\,. \]
Next, we use an idea from the proof of \cite[Theorem 7]{BogachevShkarin88}. Let $\gamma$ be a Gaussian measure on $X$ as in Lemma \ref{Lemm:measureDifferentiable}. For each $n \in \mathbb{N}$, consider the map \mbox{$F_{n}: X \rightarrow X$} defined for each $x \in X$ by
\[ F_{n}(x) = \int_{X}{\left[ F\left(x + \frac{u}{n}\right) -F\left( \frac{u}{n}\right) \right] \: d \gamma(u)} = \int_{X}{\left[F(x + u)-F(u) \right] \: d \gamma_{n}(u)} \, , \]
where $\gamma_{n}$ denotes the Gaussian probability measure defined by \mbox{$\gamma_{n}(A) = \gamma(n A)$} for each Borel subset $A$ of $X$, see \cite[Proposition 6.20]{BenyaminiLindenstrauss00}. Then,  for every $x \in X$
\[ 
\| F_{n}(x) - F(x)\| \leq  \int_{X}\|F\left( x+\frac{u}{n}\right) - F(x) \| \, d \gamma(u) \leq \| F\|_{L} \, \frac{1}{n} \, \int_{X}{\| u\| \: d \gamma(u)}, 
\]
so $F_{n}$ converges to $F$ uniformly on $X$. Using this fact together with $\| F\|_{L} = 1$, there exists a large enough natural number $N$ such that
\[ \| F_{N}\|_{L} >  1- \varepsilon \,. \]
Applying Proposition \ref{Prop:sufConditionWeakGateauxDiff}, we can moreover claim that 
\begin{equation}\label{equa:SeparableEqualityAux1} 
1- \varepsilon < \| F_{N}\|_{L} \leq \| F\|_{L}= 1 \quad , \quad w_{L}(F_{N}) \leq w_{L}(F) \leq n_{L}(X) + \varepsilon 
\end{equation}
and $F_{N}$ is weakly Gateaux differentiable at every point of $X$. Next, let us fix a pair of different points $x_{1}, x_{2} \in X$ such that
\begin{equation}\label{equa:SeparableEqualityAux2}
\| F_{N}(x_{1}) - F_{N}(x_{2})\| > (1-\varepsilon) \| x_{1} - x_{2}\|\,. 
\end{equation}
We can find $u^{\ast} \in S_{X^{\ast}}$ such that
\[ \langle u^{\ast}, F_{N}(x_{1}) - F_{N}(x_{2}) \rangle = \| F_{N}(x_{1}) - F_{N}(x_{2})\|\,.\] 
Then, applying the mean value theorem to the real function 
\[ \mathbb{R} \ni s \mapsto \langle u^{\ast}, F(x_{1} + s(x_{2}-x_{1})) \rangle ,\] we get
\begin{equation}\label{equa:SeparableEqualityAux3}
\| F_{N}(x_{1}) - F_{N}(x_{2})\|  \leq \| x_{1} - x_{2}\| \, \sup_{s \in [0,1]}{\| dF_{N}[x_{1} + s(x_{2} - x_{1})]\|}.
\end{equation}
Thus, combining \eqref{equa:SeparableEqualityAux2} and \eqref{equa:SeparableEqualityAux3} we deduce 
\[ \sup_{s \in [0,1]}{\| dF_{N}[x_{1} + s(x_{2} - x_{1})]\|} > 1-\varepsilon\,.  \]
We can then find $u \in X$ such that the derivative $T:=dF_{N}[u]$ satisfies  \mbox{$\| T\|>1-\varepsilon$}. Moreover, this linear operator satisfies 
\[ w(T) \leq w_{L}(F_{N}) \leq n_{L}(X) +  \varepsilon,\] 
by  \eqref{equa:GateauxNumericalRadius}  and \eqref{equa:SeparableEqualityAux1}. Therefore, we obtain the following upper bound
\[ n(X) \leq \frac{w(T)}{\| T\|} \leq  \frac{n_{L}(X) + \varepsilon}{1 - \varepsilon}. \]
Since $\varepsilon \in (0,1)$ is arbitrary, the result follows.
\end{proof}


\section{Real dual Banach spaces}\label{sec:realDual}

We begin by stating the main result of this section.

\begin{thm}\label{Theo:main1complemented}
For every (real) Banach space $X$ it holds that \mbox{$n_{L}(X^{\ast}) = n(X^{\ast})$}.
 \end{thm}

To prove this result, our main tool will be the method of linearization of Lipschitz maps based on \emph{invariant means}, as described in \cite{BenyaminiLindenstrauss00}. 

Let $(S,+)$ be an abelian semigroup and denote by $\ell_{\infty}(S,X)$ the Banach space of all bounded functions on $S$ with values in $X$ endowed with the supremum norm. For each $s \in S$,  denote by 
\[ \tau_{s}: \ell_{\infty}(S,X) \to \ell_{\infty}(S,X)\] 
the \emph{translation map} defined for each $f\in \ell_{\infty}(S,X)$ by
\[ \tau_{s}(f)(s') = f(s+s') \quad , \quad s' \in S \,.\]
With this notation, an \emph{$X$-valued invariant mean} on $S$ is a linear operator 
\[ \mu: \ell_{\infty}(S, X) \longrightarrow X\]
satisfying the following properties:
\begin{enumerate}
\item[(M1)] $\mu(\textbf{x}) = x$ if $\textbf{x}$ is the constant function  equal to $x \in X$,
\item[(M2)] $\| \mu \| = 1$,
\item[(M3)] $\mu \circ \tau_{s} = \mu$ for every $s \in S$.
\end{enumerate}

Every abelian semigroup $(S,+)$ has an invariant mean $\mu$ on $S$ with values in $\mathbb{R}$, see e.g. \cite[Theorem C.1]{BenyaminiLindenstrauss00}. From every such $\mu$, we can easily construct an invariant mean on $S$ with values in any dual Banach space $X^{\ast}$. We simply consider the map 
\[ \widehat{\mu}:\ell_{\infty}(S,X^{\ast}) \longrightarrow X^{\ast}\] defined for each \mbox{$f \in \ell_{\infty}(S, X^{\ast})$} by
\begin{equation}\label{equa:inducedAbelianMean} 
\langle \widehat{\mu}(f), x \rangle = \mu \left( \langle f(\cdot), x \rangle \right) \quad , \quad x \in X\,,
\end{equation}
It is easy to check that it is well-defined and satisfies the conditions \mbox{(M1)-(M3)} above using that $\mu$ also does, see \cite[Corollary C.2]{BenyaminiLindenstrauss00}.

We will also need an equivalent formulation for the numerical radius for dual Banach spaces that allows to replace elements of the bidual with elements from the predual. For its proof, we use a strategy inspired on \mbox{\cite[Proof of Lemma 5.4]{Kadets20}}.

\begin{lem}\label{Lemm:numericalRadiusDual}
Let $X$ be a (real) Banach space and let $F \in \Lip_{0}(X^{\ast})$.  If we denote for every $0 < \delta < 1$
 \begin{multline*} 
 v_{\delta}(F):=\sup\left\{ \left|\frac{\langle x, F(x_1^\ast) - F(x_2^\ast) \rangle}{\| x_{1}^\ast - x_{2}^\ast\|}\right| \colon x \in S_{X} \,, \, x_{1}^\ast \neq x_{2}^\ast \in X^\ast \,,\right.\\ \left. \frac{\langle x, x_{1}^\ast - x_{2}^\ast \rangle}{\| x_{1}^\ast - x_{2}^\ast\|} >  1-\delta  \right\}, 
 \end{multline*}
then, the numerical radius of $F$ satisfies
\begin{equation}\label{equa:numericalRadiusDual} 
w_{L}(F)  = \lim_{\delta \rightarrow 0^+}v_{\delta}(F) = \inf_{0<\delta <1} v_{\delta}(F)\,. 
\end{equation}
\end{lem}
\begin{proof}
The right-hand side equality of \eqref{equa:numericalRadiusDual} holds since $\delta \mapsto v_{\delta}(F)$ is nondecreasing. To prove the left-hand side equality, we first show that 
\begin{equation}\label{equa:numericalRadiusDualAux0}
w_{L}(F) \geq \lim_{\delta \rightarrow 0^+}v_{\delta}(F).
\end{equation}
We use the canonical identification $X \hookrightarrow X^{\ast \ast}$. Fix $0 < \delta < 1/2$ and elements \mbox{$x \in S_{X} \subset S_{X^{\ast \ast}}$} and \mbox{$x_{1}^\ast \neq x_{2}^\ast \in X^\ast$} with 
\[ \frac{\langle x, x_{1}^\ast - x_{1}^\ast \rangle}{\| x_{1}^\ast - x_{2}^\ast\|} >  1-\delta  \,. \]
 By Bishop-Phelps-Bollobás theorem \cite{Bollobas70}, there exists $x^{\ast \ast}$ and a vector $u^{\ast}$ such that
\begin{equation}\label{equa:numericalRadiusDualAux1}
\| x^{\ast \ast} - x\| \leq \sqrt{2 \delta} \, , \quad \left\| \frac{x_{1}^\ast - x_{1}^\ast }{\| x_{1}^\ast - x_{2}^\ast\|}- u^{\ast}\right\| \leq \sqrt{2 \delta} + 2\delta \, , \quad \langle x^{\ast \ast} , u^{\ast} \rangle =1\,. 
\end{equation}
Let us define $y_{1}^\ast:=x_{2}^\ast + u^\ast \|x_{1}^\ast - x_{2}^\ast \|$ and $y_{2}^\ast := x_{2}^\ast$.  Observe that $y_{1}^\ast \neq y_{2}^\ast$ and
\begin{equation}\label{equa:numericalRadiusDualAux1.5}
\| y_{1}^\ast - y_{2}^\ast\| = \| u^\ast\| \cdot \| x_{1}^\ast - x_{2}^\ast\| = \| x_{1}^\ast - x_{2}^\ast\|.
\end{equation}
Moreover, we have
\[ \langle x^{\ast \ast}, y_{1}^\ast - y_{2}^\ast \rangle = \langle x^{\ast \ast}, u^{\ast}  \rangle  \| x_{1}^\ast - x_{2}^\ast\| = \| x_{1}^\ast - x_{2}^\ast\| = \| y_{1}^\ast - y_{2}^\ast\|\,, \]
Thus, by the definition of Lipschitz numerical radius
\begin{equation}\label{equa:numericalRadiusDualAux2} 
w_{L}(F) \geq  \frac{\langle x^{\ast \ast}, F(y_{1}^\ast) - F(y_{2}^\ast) \rangle}{\| y_{1}^\ast - y_{2}^\ast\|}  \,. 
\end{equation}
On the other hand, we can estimate using \eqref{equa:numericalRadiusDualAux1.5}
\begin{align*}
& \left| \frac{\langle x^{\ast \ast}, F(y_{1}^\ast) - F(y_{2}^\ast) \rangle}{\| y_{1}^\ast - y_{2}^\ast\|} - \frac{\langle x, F(x_{1}^\ast) - F(x_{2}^\ast) \rangle}{\| x_{1}^\ast - x_{2}^\ast\|} \right| \\[2mm]
& \quad \quad \quad = \left| \frac{\langle x^{\ast \ast}, F(y_{1}^\ast) - F(y_{2}^\ast) \rangle - \langle x, F(x_{1}^\ast) - F(x_{2}^\ast) \rangle}{\| x_{1}^\ast - x_{2}^\ast\|}  \right|\\[2mm]
& \quad \quad \quad \leq \left| \frac{\langle x^{\ast \ast} - x, F(y_{1}^\ast) - F(y_{2}^\ast) \rangle }{\| x_{1}^\ast - x_{2}^\ast\|}  \right| \\[2mm]
&  \quad \quad \quad \quad \quad +\left| \frac{\langle x, F(y_{1}^\ast) - F(y_{2}^\ast) - (F(x_{1}^\ast) - F(x_{2}^\ast)) \rangle}{\| x_{1}^\ast - x_{2}^\ast\|}  \right|\\[2mm]
& \quad \quad \quad \leq \| x-x^{\ast \ast} \| \, \| F\|_{L}  \, \frac{\| y_{1}^\ast - y_{2}^\ast\|}{\| x_{1}^\ast - x_{2}^\ast\|}  + 
\left| \frac{\langle x, F(y_{1}^\ast) - F(x_{1}^\ast)  \rangle}{\| x_{1}^\ast - x_{2}^\ast\|}  \right|\\[2mm]
& \quad \quad \quad \leq \| F\|_{L}  \,\| x-x^{\ast \ast}\| +  \| F\|_{L} \,  \frac{\| x_{1}^\ast - y_{1}^\ast \|}{\|x_{1}^{\ast} - x_{2}^\ast \|}\\[2mm]
&\quad \quad \quad  = \| F\|_{L}  \,\| x-x^{\ast \ast}\| +  \| F\|_{L} \, \left\|  \frac{x_{1}^\ast - x_{2}^\ast}{\|x_{1}^{\ast} - x_{2}^\ast \|} - u^\ast \right\| \\[2mm]
& \quad \quad \quad \leq \| F\|_{L} (2\sqrt{2\delta} + 2\delta)\,.
\end{align*}
Combining the last inequality with \eqref{equa:numericalRadiusDualAux2}, we get that
\[ \left| \frac{\langle x, F(x_{1}^\ast) - F(x_{2}^\ast) \rangle}{\| x_{1}^\ast - x_{2}^\ast\|} \right| \leq \| F\|_{L} (2 \sqrt{2 \delta} + 2\delta) + w_{L}(F)\,, \]
and thus, taking supremum over $x$, $x_{1}^\ast$ and $x_{2}^\ast$ under the above conditions,
\[ v_{\delta}(F) \leq \| F\|_{L} (2 \sqrt{2 \delta} + 2\delta) + w_{L}(F)\,.  \]
Taking limit when $\delta \to 0^+$, we conclude that \eqref{equa:numericalRadiusDualAux0} holds.

The reverse inequality is easier. If we take $x^{\ast \ast} \in S_{X^{\ast \ast}}$ and vectors $x_{1}^\ast \neq x_{2}^\ast$ in $X^\ast$ with 
\[ \frac{\langle x^{\ast \ast}  , x_{1}^\ast - x_{2}^\ast\rangle}{\| x_{1}^\ast - x_{2}^\ast\|}=1,  \]
then by Goldstine's theorem, we can find for every $0<\delta <1$ an element $x \in S_{X}$ with
\[ \frac{\langle x  , x_{1}^\ast - x_{2}^\ast\rangle}{\| x_{1}^\ast - x_{2}^\ast\|} \geq 1-\delta  \quad , \quad \left|\frac{\langle x, F(x_{1}^\ast) -  F(x_{2}^\ast)\rangle}{\| x_{1}^\ast - x_{2}^\ast\|} - \frac{\langle x^{\ast \ast}, F(x_{1}^\ast) -  F(x_{2}^\ast)\rangle}{\| x_{1}^\ast - x_{2}^\ast\|}\right| < \delta\,. \]
Therefore
\[ 
\left|\frac{\langle x^{\ast \ast}, F(x_{1}^\ast) -  F(x_{2}^\ast)\rangle}{\| x_{1}^\ast - x_{2}^\ast\|}\right| \leq \delta + \left|\frac{\langle x, F(x_{1}^\ast) -  F(x_{2}^\ast)\rangle}{\| x_{1}^\ast - x_{2}^\ast\|}\right| \leq \delta + v_{\delta}(F)\,.  
\]
From this expression, taking supremum over $x^{\ast \ast}, x_{1}^\ast, x_{2}^\ast$ with the above conditions, we deduce that
\[ w_{L}(F) \leq \delta + v_{\delta}(F)\,. \]
Finally, taking again limit when $\delta \rightarrow 0^{+}$ we conclude 
\[ w_{L}(F) \leq \lim_{\delta \rightarrow 0^+}  v_{\delta}(F)\,. \]
This finishes the proof.
\end{proof}

Next, we present two results that are largely based on \cite[Theorem 7.2]{BenyaminiLindenstrauss00} and its proof. These results employ invariant means to linearize Lipschitz functions on dual Banach spaces, and we also need to include details concerning the relationship between the numerical radius of the original function and that of the resulting map.

\begin{prop}\label{Prop:FromLipschitztoLinearInvariantMean}
Let $X$ be a real Banach space and let $F \in \Lip_{0}(X^{\ast})$. For every $x^{\ast} \in X^{\ast}$ consider the map $F_{x^\ast}:X^\ast \to X^\ast$ given by
\[ F_{x^\ast}(y^{\ast}) = F(x^{\ast} + y^\ast) - F(y^\ast) \quad , \quad y^{\ast} \in X^{\ast}. \]
Then, this map is bounded. Moreover, let $Z$ be a vector subspace of $X^{\ast}$ and denote the restriction of $F_{x^\ast}$ to $Z$ also by $F_{x^\ast}$. Fix a scalar invariant mean $\mu$  on the abelian group $(Z,+)$, and let $\widehat{\mu}:\ell_{\infty}(Z,X^{\ast}) \to X^{\ast}$ be the associated invariant mean with values in $X^{\ast}$ defined in \eqref{equa:inducedAbelianMean}.  Then, the map \mbox{$T: X^{\ast} \longrightarrow X^{\ast}$} given by
\[ T(x^{\ast}) = \widehat{\mu}\left(  F_{x^\ast}  \right)\,,\quad x^{\ast} \in X^{\ast} \]
satisfies: 

\begin{enumerate}
\item[($i$)] $T$ is Lipschitz with $T(0)=0$ and $\| T\|_{L} \leq \| F\|_{L}$,

\item[($ii$)] $T(x^{\ast} + z^{\ast}) = T(x^{\ast}) + T(z^{\ast})$ for every $x^{\ast} \in X^{\ast}$ and $z^{\ast} \in Z$,

\item[($iii$)] $w_{L}(T) \leq w_{L}(F)$.
\end{enumerate}
\end{prop}

\begin{proof}
The fact that $F_{x^\ast}$ is bounded for every $x^{\ast} \in X^\ast$ is clear:
\[ \sup_{y^{\ast} \in X^\ast}\| F_{x^\ast}(y^\ast)\| = \sup_{y^{\ast} \in X^\ast}\| F(x^\ast + y^\ast) - F(y^\ast)\| \leq \| F\|_{L} \, \| x^\ast\|\,. \]
Next, let us check ($i$). Since $\widehat{\mu}$ is linear, for every $x^{\ast}, y^{\ast} \in X^{\ast}$ we have
\[  T(x^{\ast})-T(y^{\ast})  = \widehat{\mu}\left( F_{x^\ast} - F_{y^\ast} \right) \,.\]
Using that $\|\widehat{\mu} \|=1$ by (M2), we can estimate
\begin{align*} 
\| T(x^{\ast}) - T(y^{\ast})\|
& \leq \| F_{x^\ast} - F_{y^{\ast}} \|_{\infty}\\[2mm] 
&  = \sup_{z^{\ast} \in Z}\| F_{x^\ast}(z^{\ast}) - F_{y^{\ast}}(z^{\ast})  \|\\[2mm]
&  = \sup_{z^{\ast} \in Z}\| F(x^\ast + z^{\ast}) -  F(y^\ast + z^{\ast})  \|\\[2mm]
& \leq \| F\|_{L} \, \| x^{\ast} - y^{\ast}\|\,. 
\end{align*}
This yields that $T$ is Lipschitz with $\|T\|_{L} \leq \| F\|_{L}$. The fact that $T(0)=0$ is obvious.

 To see property ($ii$), let us fix $x^\ast \in X^\ast$ and $z^\ast \in Z$. Notice that for every $y^\ast \in Z$
\begin{align*} 
F_{x^\ast+z^\ast}(y^\ast) - F_{z^\ast}(y^\ast)
& = F(x^\ast+z^\ast + y^\ast) - F(z^\ast + y^\ast)\\[2mm] 
& = F_{x^\ast}(y^\ast + z^\ast)\\[2mm]
& =\tau_{z^\ast}(F_{x^{\ast}})(y^\ast) \,,
\end{align*}
where $\tau_{z^\ast}$ is the translation map on $\ell_{\infty}(Z,X)$. This identity can be rewritten as
\[ F_{x^\ast+z^\ast} - F_{z^\ast} = \tau_{z^\ast}(F_{x^{\ast}}). \]
Applying then $\widehat{\mu}$ on this equality, and using (M3), we get
\begin{align*}
T(x^{\ast}+z^{\ast})-T(z^{\ast})  
 & =\widehat{\mu}\left( F_{x^\ast+z^\ast} - F_{z^\ast} \right)  \\[2mm]
 & = \widehat{\mu}\left(\tau_{z^\ast}(F_{x^\ast})\right)     \\[2mm]
 & = \widehat{\mu}\left(F_{x^\ast}\right)  \\[2mm]
 & = T(x^{\ast})\,.
\end{align*}
This proves the second item.

Finally, let us check ($iii$). Following Lemma \ref{Lemm:numericalRadiusDual}, it is enough to prove that for every $0 < \delta <1$ it holds that $v_{\delta}(T) \leq v_{\delta}(F)$. To see this, let us take $x \in S_{X}$ and $x_{1}^\ast \neq x_{2}^\ast$ in $X^{\ast}$ satisfying
\[ \frac{\langle x, x_{1}^\ast - x_{2}^\ast\rangle}{\| x_{1}^\ast - x_{2}^\ast\|} > 1- \delta \,.\]
Observe that the same inequality holds if we replace $x_{1}^\ast$ and $x_{2}^\ast$ with $x_{1}^\ast +z^\ast$ and $x_{2}^\ast +z^\ast$, respectively, for any $z^\ast \in Z$. Then, we can estimate
\begin{align*} 
\frac{|\langle x, T(x_{1}^{\ast}) - T(x_{2}^{\ast})\rangle|}{\| x_{1}^\ast - x_{2}^\ast\|} 
& =  \frac{|\langle x, \widehat{\mu}\left(F_{x_1^\ast} - F_{x_2^\ast}\right)\rangle|}{\| x_{1}^\ast - x_{2}^\ast\|} \\[2mm]
& = \frac{|\mu\big(  \langle x,  F_{x_{1}^{\ast}} -F_{x_{2}^{\ast}} \rangle \big)|}{\| x_{1}^\ast - x_{2}^\ast\|}\\[2mm]
& \leq \frac{\| \mu \| \sup_{z^{\ast} \in Z} |\langle x,  F_{x_{1}^{\ast}}(z^\ast) -F_{x_{2}^{\ast}}(z^\ast) \rangle|}{\| x_{1}^\ast - x_{2}^\ast\|}\\[2mm]
& =  \sup_{z^{\ast} \in Z} \frac{|\langle x,  F(x_{1}^{\ast}+z^{\ast}) -F(x_{2}^{\ast}+z^{\ast}) \rangle|}{\| (x_{1}^\ast+z^\ast) - (x_{2}^\ast+z^\ast)\|}\\[2mm]
& \leq v_{\delta}(F)\,.
\end{align*}
Taking then supremum over all $x$ and $x_{1}^\ast, x_{2}^\ast$  with the above condition, we conclude that $v_{\delta}(T) \leq v_{\delta}(F)$, finishing the proof.
\end{proof}

The following result follows the ideas from \cite[Theorem 7.2]{BenyaminiLindenstrauss00}.
 
\begin{cor}\label{Coro:FromLipschitztoLinearInvariantMean}
Let $X$ be a Banach space and let $Z$ be a subspace of $X^{\ast}$. Then, for each $F \in \Lip_{0}(X^{\ast})$ with $F|_{Z}:Z \longrightarrow X^{\ast}$ being linear, there exists a bounded linear operator $T:X^{\ast} \longrightarrow X^{\ast}$ such that
\[ T|_{Z} = F|_{Z}\,, \quad \| T\| \leq \| F\|_{L} \quad \mbox{and} \quad w(T) \leq w_{L}(F). \]
\end{cor} 
\begin{proof}
Let us fix a real-valued invariant mean $\mu$ on $(Z,+)$, and let  $\widehat{\mu}$ be the associated $X^\ast$-valued invariant mean. Following Proposition \ref{Prop:FromLipschitztoLinearInvariantMean}, let us define from $\widehat{\mu}_{Z}$ and $F$ a new map  \mbox{$G: X^{\ast} \longrightarrow X^{\ast}$} by
\[ G(x^{\ast}) = \widehat{\mu}\left(  F_{x^\ast}  \right)\,,\quad x^{\ast} \in X^{\ast}\,. \]
By the aforementioned proposition, $G$ is a Lipschitz map satisfying $G(0)=0$, $\| G\|_{L} \leq \| F\|_{L}$, $w_{L}(G) \leq w_{L}(F)$ and
\begin{equation}\label{equa:GquasiLinear} 
G(x^{\ast} + z^{\ast}) = G(x^\ast) + G(z^\ast) \quad \mbox{for every }x^{\ast} \in X^\ast , z^\ast \in Z\,. 
\end{equation}
Moreover, for every $z^\ast \in Z$, since $F|_{Z}$ is linear, it holds that $F_{z^\ast}$ is a constant function on $Z$ equal to $F(z^\ast)$. Thus, property (M1) yields that 
\[ G(z^\ast) = \widehat{\mu}(F_{z^\ast}) = F(z^\ast) \quad \mbox{for every } z^\ast \in Z\,, \]
or equivalently, $G|_{Z} = F|_{Z}$. However, it is not clear that $G$ is linear on $X^\ast$. To achieve this property, we need to iterate the same argument once more.

Let us then fix a real-valued invariant mean $\nu$ on $(X^\ast,+)$, and consider its associated $X^\ast$-valued invariant mean $\widehat{\nu}$. We follow again Proposition \ref{Prop:FromLipschitztoLinearInvariantMean} to define from $\widehat{\nu}$ and $G$ a new map \mbox{$T:X^{\ast} \longrightarrow X^{\ast}$} given by
\[ T(x^{\ast}) =\widehat{\nu}(G_{x^\ast}) \quad , \quad x^{\ast} \in X^{\ast}. \]
By the aforementioned proposition, this is again a Lipschitz map satisfying 
\[ \| T\|_{L} \leq \|G\|_{L} \leq \| F\|_{L} \quad , \quad w_{L}(T) \leq w_{L}(G) \leq w_{L}(F) \]
and moreover
\[ T(x^{\ast} + y^{\ast}) = T(x^\ast) + T(y^\ast) \quad \mbox{for every }x^{\ast} \in X^\ast , y^\ast \in X^\ast\,.\]
This last property, together with the fact that is Lipschitz, yields that $T$ is actually a (bounded) linear operator, which thus satisfies $\| T\| \leq \| F\|_{L}$ and $w(T) \leq w_{L}(F)$.

Moreover, using \eqref{equa:GquasiLinear}, we have that for every $z^\ast \in Z$ the map $G_{z^\ast}$ is a constant function equal to $G(z^\ast)$. Hence, property (M1) yields that 
\[ T(z^\ast) = \widehat{\nu}(G_{z^\ast}) = G(z^\ast) \quad \mbox{for every } z^\ast \in Z\,. \]
This implies that $T|_{Z} = G|_{Z} = F|_{Z}$, and the result follows.
\end{proof}

We can now finish the section by proving the main result.

\begin{proof}[Proof of Theorem \ref{Theo:main1complemented}]
We only need to show that $n(X^\ast) \leq n_{L}(X^\ast)$. For that, let us fix $\varepsilon \in (0,1)$ and take $F \in \Lip_{0}(X^{\ast})$ such that $\| F\|_{L} = 1$ and $w_{L}(F) < n_{L}(X) + \varepsilon$. 

We claim that there exists a separable subspace $Z \subset X^{\ast}$ such that $F(Z) \subset Z$ and the Lipschitz norm of its restriction  $\| F|_{Z}\|_{L} = 1$. Indeed, by the definition of the Lipschitz norm, we can find two sequences $(x_{n})_{n}$ and $(y_{n})_{n}$ such that $x_{n} \neq y_{n}$ and
\[ \frac{\| F(x_{n}) - F(y_{n})\|}{\| x_{n} - y_{n}\|} \geq 1 - \frac{1}{n} \quad \mbox{ for every } n \in \mathbb{N}\,. \]
Let $Z_{0}$ be (separable) the closed subspace generated by both sequences. Note that the restriction $F|_{Z_0}: Z_{0} \to X^\ast$ satisfies $\| F|_{Z_0} \|_{L} = 1$. However, it may happen that $F(Z_0) \nsubseteq Z_0$. To address this, we recursively construct an increasing sequence of separable closed subspaces $(Z_{n})_{n \geq 0}$, by taking 
\[
Z_{n+1} := \overline{\text{span}} \left( Z_n \cup F(Z_n) \right).
\]
Finally, we take $Z$ as the closure of the subspace $\cup_{n \geq 0} Z_n $. This subspace is separable, closed, satisfies $F(Z) \subset Z$, and $\| F|_{Z} \|_{L} = 1$, as desired.

Since $Z$ is separable, we can fix a Gaussian measure $\gamma$ on $Z$, and consider the sequence of functions
\[ F_{n}:X^\ast \longrightarrow X^\ast \quad , \quad F_{n}(x^\ast) = \int_{Z} F\left( x^\ast+ \frac{z^\ast}{n}\right) d \gamma(z^\ast)\,. \]
Note that they also satisfy $F_{n}(Z) \subset Z$ for every $n$. Arguing as in the proof of Theorem \ref{Theo:EqualityNumIndexesSeparable}, for every $x^\ast \in X^\ast$ we have
\[ \| F_{n}(x^\ast) - F(x^\ast)\| \leq \|F \|_{L} \, \frac{1}{n} \int_{Z} \| z^\ast\| d \gamma(z^\ast)\,, \]
so that $\lim_{n}{\| F_{n}(x^\ast) - F(x^\ast)\|} = 0$ uniformly on $X$, and thus 
\[  \lim_{n}{w_{L}(F_{n})} = w_{L}(F) \quad , \quad \lim_{n}{\| F_{n}\|_{L}} = \| F\|_{L} \,. \]
As in the aforementioned theorem, we can take a natural number $N$ large enough so that 
\[ 1- \varepsilon < \| F_{N}\|_{L} \leq 1 \quad , \quad n_{L}(X) - \varepsilon < w_{L}(F_{N}) \leq n_{L}(X) \,. \]
Let us denote by $G: Z \longrightarrow Z$ the restricted map $G = F_{N}|_{Z}$. Observe that $G$ is weakly Gateaux differentiable at every $z^\ast \in Z$ and satisfies
\[ 1-\varepsilon \leq \| G\|_{L}  \leq 1 \quad , \quad n_{L}(X) - \varepsilon < w_{L}(G) \leq n_{L}(X). \]
Moreover, arguing as in the proof of Theorem \ref{Theo:EqualityNumIndexesSeparable}, there exists an element $z_{0}^{\ast} \in S_{Z} \subset S_{X^{\ast}}$ such that
\[ \|dG[z_{0}^\ast] \| \geq 1- \varepsilon\,.\]
Let $\mathcal{U}$ be an ultrafilter on $(0,+\infty)$ containing the subsets of the form $(0, \varepsilon)$ for every $\varepsilon>0$, and define the function
\[ \psi: X^{\ast} \longrightarrow X^{\ast} \quad , \quad \langle x,\psi(x^{\ast}) \rangle = \lim_{t, \mathcal{U}} \frac{\langle x , F_{N}(z_{0}^{\ast}+tx^{\ast}) - F_{N}(z_{0}^{\ast}) \rangle}{t} \]
where $\lim_{t, \mathcal{U}}$ stands for the limit on $t \in (0, \infty)$ with respect to the ultrafilter $\mathcal{U}$. Note that this limit actually exists, since the function on the right is bounded on $t$:
\[ |\langle x, F_{N}(z_{0}^\ast + t x^\ast) - F_{N}(x_{0}^\ast) \rangle \rangle| \leq \| x \| \cdot \|F_{N} \|_{L} \cdot |t| \cdot \| x^\ast\|\,.\]
Moreover, we require that $\mathcal{U}$ contains all subsets of the form $(0, \varepsilon)$, so that whenever the following (ordinary) limit 
\[ \lim_{t \rightarrow 0^+} \frac{\langle x , F_{N}(z_{0}^{\ast}+tx^{\ast}) - F_{N}(z_{0}^{\ast}) \rangle}{t} \]
exists, its value coincides with $\langle x,\psi(x^{\ast}) \rangle$. This property will be useful later. We claim that the map $\psi$ satisfies the following properties:
\begin{enumerate}
\item[($i$)] $\psi$ is a Lipschitz map with $\| \psi\|_{L} \leq 1$. Indeed, for every $x^{\ast}_{1},x_{2}^{\ast} \in X$
\begin{align*} 
\| \psi(x_{1}^{\ast}) - \psi(x_{2}^{\ast}) \| 
& \leq \sup_{x \in S_{X}} \lim_{t, \mathcal{U}} \frac{|\langle x, F_{N}(z_{0}^{\ast} + tx_{1}^{\ast}) - F_{N}(z_{0}^{\ast}+tx_{2}^{\ast})\rangle|}{t}\\[2mm] 
& \leq \| F\|_{L} \, \|x_{1}^{\ast}-x_{2}^{\ast} \| \,.  
\end{align*}

\item[($ii$)] For each $z^{\ast} \in Z$ we have $ \psi(z^{\ast}) = dG[z_{0}^{\ast}](z^{\ast})$, so in particular, $\psi|_{Z}$ is linear. Indeed, since the ultrafilter contains subsets of the form $(0, \varepsilon)$ and $G$ is weakly Gateaux differentiable at $x_{0}^\ast$, then  for every $x \in X \subset X^{\ast \ast}$
\begin{align*}
\langle  x, \psi(z^{\ast})  \rangle 
& = \lim_{t, \mathcal{U}} \frac{\langle x , F_{N}(z_{0}^{\ast}+tz^{\ast}) - F_{N}(z_{0}^{\ast}) \rangle}{t}\\[2mm]
& = \lim_{t \to 0^+} \frac{\langle x , F_{N}(z_{0}^{\ast}+tz^{\ast}) - F_{N}(z_{0}^{\ast}) \rangle}{t} \\[2mm]
& = \langle x, dG[z_{0}^{\ast}](z^{\ast}) \rangle
\end{align*}
for every $x \in X$, which proves the claim.
\item[($iii$)] $w_{L}(\psi) \leq w_{L}(F_{N}) \leq n_{L}(X^{\ast}) + \varepsilon$. To check this claim, it suffices to prove that
\[ v_{\delta}(\psi) \leq v_{\delta}(F_N) \quad \mbox{for every } 0<\delta < 1\,,\]
and apply Lemma \ref{Lemm:numericalRadiusDual}. Let us then fix $0<\delta<1$, and take $x \in S_{X}$ and $x_{1}^\ast, x_{2}^\ast  \in X^{\ast}$ satisfying $x_{1}^\ast \neq x_{2}^{\ast}$ and
\[
\frac{\langle x, x_{1}^\ast - x_{2}^\ast\rangle}{\| x_{1}^\ast - x_{2}^\ast\|}>1-\delta\,.
\]
Note that the same inequality holds if we replace $x_{1}^\ast$ and $x_{2}^\ast$ with $tx_{1}^\ast + z_{0}^\ast$ and $tx_{2}^\ast + z_{0}^\ast$, respectively, for any $z_{0}^\ast \in X$ and $t >0$. Therefore. we can estimate
\begin{align*}
\frac{\langle x, \psi(x_{1}^\ast) - \psi(x_{2}^\ast) \rangle}{\| x_{1}^\ast - x_{2}^\ast \|} 
& = \lim_{t, \mathcal{U}}\frac{\langle x, F_{N}(z_{0}^\ast + t x_{1}^\ast) - F_{N}(z_{0}^\ast +t x_{2}^\ast) \rangle}{t \|x_{1}^\ast  - x_{2}^\ast\|}\\[2mm]
& =\lim_{t, \mathcal{U}}\frac{\langle x, F_{N}(z_{0}^\ast + t x_{1}^\ast) - F_{N}(z_{0}^\ast +t x_{2}^\ast) \rangle}{\| (z_{0}^\ast + t x_{1}^\ast) - (z_{0}^\ast + t x_{2}^\ast) \|}\\[2mm]
& \leq v_{\delta}(F_N)\,.
\end{align*}
Taking supremum over all $x, x_{1}^\ast, x_{2}^\ast$ under the above properties, we conclude that $v_{\delta}(\psi) \leq v_{\delta}(F_N)$, and the result follows.
\end{enumerate}
\noindent Next, we apply Corollary \ref{Coro:FromLipschitztoLinearInvariantMean} to $\psi$ and $Z$, to get a bounded linear operator $T: X^\ast \to X^\ast$ such that 
\[ T|_{Z} = \psi|_{Z} \quad , \quad \| T\| \leq \| \psi\|_{L} \leq 1 \quad , \quad w(T) \leq w_{L}(\psi)  < n_{L}(X^{\ast}) + \varepsilon. \]
In particular,
\[ T(z^{\ast}) = \psi(z^{\ast}) = dG[z_{0}^{\ast}](z^{\ast}) \quad \mbox{ for every } \quad z^{\ast} \in Z,\] 
so that
\[ \| T\| \geq \| dG[z_{0}^{\ast}]\| \geq 1- \varepsilon\,. \]
Finally, combining these relations, we deduce that
\[ n(X^{\ast}) \leq \frac{w(T)}{\| T\|} \leq \frac{n_{L}(X) + \varepsilon}{1-\varepsilon}\,. \]
Since $\varepsilon \in (0,1)$ was arbitrary, we conclude the result.
\end{proof}

\bigskip

\subsection*{Acknowledgment}

The author is grateful to V. Kadets, M. Martín, and J. Merí for helpful conversations and for the stimulating mathematical environment that contributed to the development of the ideas presented in this article. This work is based on some notes prepared by the author during the final stages of his PhD, around 2017, which were set aside for some time. Considering that the material may be of interest, he has now decided to make it available.



\end{document}